\documentclass[oneside,reqno]{amsart}
\pdfoutput=1
\usepackage{graphicx}

\usepackage{subfigure}
\usepackage{amsmath}
\usepackage{mathrsfs} 
\usepackage{amsfonts}
\usepackage{amssymb}%
\usepackage{amsthm}
\usepackage{amscd}
\usepackage{verbatim}
 \usepackage{mathabx}
 \usepackage{dsfont}
 \usepackage[all]{xypic} 
\usepackage{endnotes}
\usepackage{enumerate}
\usepackage{mathtools}
\usepackage{comment}
\usepackage{tikz}
\usepackage{color}

\newcommand{\mcA}{\mathcal{A}}

\newcommand{\mcF}{\mathcal{F}}

\newcommand{\mcL}{\mathcal{L}}

\newcommand{\mcR}{\mathcal{R}}

\newcommand{\C}{\mathbb{C}}

\newcommand{\N}{\mathbb{N}}

\newcommand{\bfd}{\mathbf{d}}

\newcommand{\ot}{\otimes}

\newcommand{\GL}{\operatorname{GL}}

\newcommand{\Vect}{{\rm Vect}}
\newcommand{\fdVect}{\Vect}
\newcommand{\rank}{\operatorname{rank}}

\newcommand{\unit}{\mathds{1}}
\newcommand{\Ob}{\text{Ob}}

\newcommand{\Hom}{\text{Hom}}

\newcommand{\tn}{\textnormal}

\newcommand{\End}{\textnormal{End}}
\newcommand{\ott}{\bigotimes}

\newcommand{\op}{\oplus}
\newcommand{\bop}{\bigoplus}

\newcommand{\rhu}{\overset{\rightharpoonup}}
\newcommand{\lhu}{\overset{\leftharpoonup}}
\newcommand{\into}{\hookrightarrow}

\newcommand{\GLd}{\textnormal{$\textnormal{GL}_{\textbf{d}}$}}

\newcommand{\onto}{\twoheadrightarrow}

\theoremstyle{plain}
\newtheorem{theorem}{Theorem}[section]
\newtheorem{corollary}[theorem]{Corollary}
\newtheorem{proposition}[theorem]{Proposition}
\newtheorem{lemma}[theorem]{Lemma}

\theoremstyle{definition}
\newtheorem{defn}[theorem]{Definition}
\theoremstyle{definition}

\theoremstyle{definition}

\theoremstyle{definition}

\theoremstyle{definition}

\theoremstyle{definition}
\newtheorem{fact}{Fact}

\usepackage{hyperref}

\title{A Finite-Tame-Wild Trichotomy Theorem for Tensor Diagrams}
\author{Jacob Turner${}^\dagger$}
\thanks{${}^\dagger$ Korteweg-de Vries Institute for Mathematics, University of Amsterdam, 1098 XG Amsterdam, Netherlands.}

\begin{document}

 \begin{abstract}
 In this paper, we consider the problem of determining when two tensor networks are equivalent under a heterogeneous change of basis. In particular, to a string diagram in a certain monoidal category (which we call tensor diagrams), we formulate an associated abelian category of representations. Each representation corresponds to a tensor network on that diagram. We then classify which tensor diagrams give rise to categories that are finite, tame, or wild in the traditional sense of representation theory. For those tensor diagrams of finite and tame type, we classify the indecomposable representations. Our main result is that a tensor diagram is wild if and only if it contains a vertex of degree at least three. Otherwise, it is of tame or finite type.
 \end{abstract}
  \maketitle
  \begin{center}{\small {\sc Keywords:} Tensor Networks, Representation Theory, Wild Representations, Tame Representations}
\end{center}
\section{Introduction}

Tensor networks were first introduced by Roger Penrose in the 1970's as a way of studying systems in physics \cite{penrose1971applications}. They have exploded in the past few decades as an active area of research with various different applications. In quantum physics, they generalize the quantum circuit model in quantum computing, are used in the study of approximating ground states of Hamiltonians \cite{white1992density,verstraete2004renormalization,orus2014practical}, appear as models in statistical physicals, including the Potts and Ising models \cite{delaharpe1993graph}, and appear as invariants of quantum entanglement \cite{turner2017complete,grassl1998computing,hero2009stable,johansson2012topological,kraus2010local,MR2039690,luque2006algebraic,macicazek2013many}.

In algebraic complexity theory, tensor network contraction is a $\#\mathsf{P}$-hard problem which can be used to design algorithms for combinatorial counting problems. Leslie Valiant found a certain class of tensor networks that always had a polynomial time evaluation, which he called holographic algorithms \cite{ValiantFOCS2004,MR2120307,landsberg2012holographic,morton2010pfaffian}. With this framework, Valiant found new polynomial time algorithms not known to have any \cite{valiant2006accidental}. Later attempts at generalization also found new polynomial time algorithms for computing partition functions and Tutte polynomials \cite{morton2015generalized, morton2015computing}. These algorithms have lead to interests in dichotomy theorems regarding the hardness of tensor network contraction \cite{bulatov2005complexity,dyer2010effective,cai2011non,cai2012complexity,cai2010holographic,cai2013complete,cai2015holant}.
Other notable examples of tensor networks are graphical models in algebraic statistics and phylogenetic tree models in evolutionary biology. 

Given a tensor network, there is natural set of basis changes that typically leave what is considered most important about the tensor network invariant. This group action is sometimes called a heterogeneous change of basis. When tensor network states are used to approximate the ground states of Hamiltonians, these states can be viewed as covariant maps, with respect to this group action, from the representation space of some tensor diagram to a variety of states. Furthermore, the variety of such states carries a similar group action which is poorly understood. Studying the orbits of these tensor networks under the natural group action is a central problem in this area.

Tensor networks are themselves invariant polynomials that are invariant under the induced group action (cf. \cite{biamonte2013tensor}). Invariant theory and the study of entanglement have a close relationship \cite{PhysRevA.62.062314,luque2007unitary,klyachko2002coherent}. Recently, tensor networks we shown to give complete sets of invariants of density operators under local unitary equivalence \cite{turner2017complete}, a specific case of determining the orbits of certain tensor networks under heterogeneous changes of basis. This makes them an important measure of entanglement.

In the study of tensor networks as circuits, the question of which tensor networks are equivalent to those with a known polynomial time evaluation procedure is of great interest. This includes the question of both classifying and recognizing Holant problems, holographic circuits, and matchgates (cf. \cite{margulies2016polynomial,turner2017tensors}). Recognizing a holographic circuit in an arbitrary basis is still an open problem as they are formulated with respect to a specific basis. 

In SLOCC paradigm of studying entanglement as well as the study of tensor rank in applied algebraic geometry both consider the the natural group action on very simple types of tensor networks. In this setting, determining the orbits is already known to be intractable in most cases \cite{belitskii2003complexity}. More specifically, it was shown that determining the orbits of trivalent tensors under the natural group action is as difficult as determining simultaneous similarity of pairs of matrices.

As such, the orbit classification problem of trivalent tensors is a wild problem and is as difficult to solve as classifying the orbits of any finite dimensional algebra. In this paper, we expand the classification not only to tensors of every arity, but to all tensor diagrams. A tensor diagram is a string diagram in the graphical language of free monoidal category. We will consider all tensor networks that can be associated to a given tensor diagram and give a complete classification of which tensor diagrams have finitely many orbits in every dimension, a finite number of one parameter families in every dimension, and which are wild.

We first reinterpret tensor networks on a tensor diagram as an abelian category of representations of a combinatorial object, similar to the approach used in quiver theory. We then construct explicit inclusions of wild subcategories in the category of representations of tensor diagrams of wild type. Otherwise, as we are in an abelian category, we have a notion of indecomposable objects, and for those tensor diagrams of finite and tame type, we classify the indecomposable representations.

Our main theorem states that a tensor diagram is wild if and only if it have vertex of degree at least three and otherwise it is of tame or finite type. Furthermore, we state explicitly which tensor diagrams are of tame type and which are of finite type.

\section{Preliminaries}

We first give a pedestrian definition of tensor diagram and a representation of such an object. First we define a \emph{semi-graph} to be a finite collection of vertices and edges $(V,E)$, where each edge is incident to \emph{at most} two vertices. An edge may also only be incident to zero or one vertex. We call such an edge \emph{dangling}. If a tensor diagram has no dangling wires, it is called \emph{closed}.

\begin{defn}
 A tensor diagram $T=(V,E)$ is a directed semi-graph.
\end{defn}

In the context of tensor diagrams, instead of saying ``edges'', we will instead use the term ``wire'' instead. This is more consistent with the literature on tensor networks, although sometimes the term string is used instead. A \emph{subdiagram} of a tensor diagram $(V,E)$ is a pair $V'\subseteq V$ and $E'\subseteq E$ where $(V',E')$ is a semi-graph with the inherited orientation. A subdiagram $(V',E')\subseteq (V,E)$ is called \emph{induced} if for every $v,w\in V'$, the set of wires between $v,w$ in $E'$ is the same as the set of wires between them in $E$.

\begin{defn}
 Let $T=(V,E)$ be a tensor diagram. A representation of $T$, denoted $R(T)$ is a pair of maps: For every $e\in E$, $R(e)$ is a vector space over a given field $k$.
 For each vertex $v$, let $\rhu{N}(v)$ denote set of incoming wires and $\lhu{N}(v)$ the set of outgoing wires. Then $R(v)$ is a matrix in $$\Hom\bigg(\ott_{A\in\rhu{N}(v)}{A},\ott_{B\in\lhu{N}(v)}{B}\bigg).$$ We call a representation of a tensor diagram a \emph{tensor network}.
\end{defn}

As in quiver theory, representations can be completely defined in a combinatorial manner but benefit by viewing them in terms of categories. Finite quivers correspond to finite categories, and tensor diagrams will be morphisms in a finitely generated monoidal category, which we will define shortly.

We assume in this paper that the reader is familiar with compact closed monoidal categories. In particular we consider free compact closed monoidal categories; such categories are always equivalent to a category of diagrammatic languages whose morphisms are in bijection with labeled tensor diagrams. Furthermore, we shall need basic facts about cup and cap morphisms and symmetric braiding morphisms in a monoidal category. Excellent treatments of this subject can be found in \cite{maclane1998categories,joyalgeometry,selinger2010survey}.

Let us consider any finitely generated free compact closed monoidal category $\mathcal{F}$. By finitely generated, we mean that $\Ob(\mathcal{F})$ is generated by taking all tensor products of a finite number of \emph{atomic} objects that cannot be written as a tensor product of other objects, with the obvious exception of the natural isomorphisms $\unit\ot A\to A$, $A\ot \unit\to A$, for any object $A$ and $\unit$ the monoidal unit in $\mathcal{F}$.

We note that our theory need not consider a finitely generated monoidal category; however, since a tensor diagram can only represent a finite number of objects and morphisms, this is sufficient for our purposes. Indeed, we could consider full subcategories of a monoidal category induced by the objects in a tensor diagram and this would be equivalent. Using finitely generated monoidal categories simplifies the technical aspect of monoidal signatures that we will need.

Since we assumed it was free, the category $\mcF$ can be visualized in terms of diagrams, as previously mentioned (glossing over the technicality that we may need to replace $\mcF$ with an equivalent category). We take the convention that primal objects in $\mcF$ are given by right oriented wires, dual objects by left oriented wires, and that tensor products shall be taken vertically in order from top to bottom. As usual, the unit $\unit$ is denoted by empty space.

Let $Ob_{\mcA}(\mcF)$ denote the atomic \emph{primal} objects of $\mcF$. We will assume that these objects have a labeling $\ell:\mcL\to Ob_{\mcA}(\mcF)$, where $\mcL$ are the labels and we only demand that the unit be labeled $\unit$. Any other object in $\mcF$ inherits a label from $\ell$ as follows: if some object can be expressed as tensor product of $k$ atomic objects in $\mcF$, then its label is $A_1\ot\cdots\ot A_k$, where $A_i$ is the label of the $i$th object as specified by $\ell$. If if $A$ is the label of a primal object, then $A^*$ is the label of its dual object. We note that one can also form $A^{**}$, but this is naturally isomorphic to $A$ so we do not need this label. We similarly define labels for morphisms in $\mcF$. Since $\mcF$ is a free category, we have no non-trivial equivalence of labels. Together, these labels form a monoidal signature; for a more precise definition of monoidal signatures, see \cite{joyal1988planar,joyal1991geometry}.

Monoidal signatures formally allow us to use variables to describe objects and morphisms in $\mcF$ where the variables can take any value in the monoidal signature. In a tensor diagram, one may label wires with an object variable and vertices with morphism variables. We remark that one does not typically label the unit in a tensor diagram. In this work, we shall consider all labels, both on wires and on vertices, as distinct. This coincides with the definition given above.

 We note that the wire denoting the object $A$ is naturally associated with the identity morphism on $A$. A tensor diagram also forms a morphism from the tensor product of the dangling dual wires to the tensor product of the dangling primal wires. If there are no dangling wires, the tensor diagram is a morphism in $\Hom(\unit,\unit)$. We now present and equivalent, but more abstract definition of a representation of a tensor diagram a.k.a. a tensor network. 

\begin{defn}
 Let $\mcF$ be a finitely generated free compact closed monoidal category and $\Psi$ be a strict monoidal functor from $\mcF$ into the category $\fdVect_k$ (the category of finite dimensional vector spaces over $k$) for some field $k$. A \emph{tensor network} is the image of a morphism of $\mcF$ under the map $\Psi$.
\end{defn}

\vspace{.3cm}
\noindent{\textbf{Organization of the paper}}. In Section \ref{sec:repspace}, we identity the space of representations of fixed dimension of a tensor diagram and show that it forms an abelian category. We also discuss the natural group action that arises from isomorphism classes of representations. Then in Section \ref{sec:forbsuddiag}, we formulate the notions of a finite, tame, and wild tensor diagrams in terms of their category of representations. We then explicitly show that certain tensor diagrams are of wild type. In Section \ref{sec:inclusion}, we show that the category of representations of a tensor diagram generally contains the category of representations of any subdiagram. Lastly, in Section \ref{sec:tameandfinite}, we classify those tensor diagrams which are tame and finite and classify their indecomposable representations.

\section{The Space of Representations of a Tensor Diagram}\label{sec:repspace}

Given a tensor diagram $T$ with wires $E$ and vertices $V$, a \emph{dimension vector} is a tuple of non-negative integers $\textbf{d}=(d_1,\dots,d_{|E|})$. From here on out, we assume that our wires are labeled by natural numbers. We say that representation of $T$ has dimension vector $\textbf{d}$ if the vector space associated to the wire $i$ has dimension $d_i$, for all $i\in[|E|]$. Given such a representation, we can consider the wire $i$ to have the vector space $k^{d_i}$ associated to it. If we do this, the set of all representation of $T$ with dimension vector $\textbf{d}$ form a vector space, namely

$$\mcR_{\textbf{d}}(T):=\bop_{v\in V}{\Hom\bigg(\ott_{i\in\rhu{N}(v)}{k^{d_i}},\ott_{j\in\lhu{N}(v)}{k^{d_j}}\bigg)}.$$

However, by choosing the vector space on the $i$th wire to be $k^{d_i}$, we have made an implicit choice of basis unnecessarily. We note that we could change this basis and have an isomorphic representation. This induces an action of the following group
$$\GLd(T):=\bop_{i\in E}{\GL(d_i,k)}$$ on $\mcR_{\textbf{d}}(T)$ acting by a change of basis on each wire. To be very concrete, given an element $g=(g_1,\dots,g_{|E|})\in\GLd(T)$, the induced action on the matrix associated to a vertex $v$, call it $M_v$, is as follows:
$$g.M_v=\bigg(\ott_{j\in\lhu{N}(v)}{g_j}\bigg)M_v\bigg(\ott_{i\in\rhu{N}(v)}{g_i^{-1}}\bigg).$$ The orbits of $\GLd(T)$ on $\mcR_{\textbf{d}}(T)$ are precisely the isomorphism classes of representations of $T$ with dimension vector $\bfd$. 

We denote the space of all representations of $T$ by $\mcR(T)$. We will see that this space has the structure of an abelian category in addition to being a closed compact monoidal category. This category will be naturally graded by dimension vectors $\bfd$.

We now define the morphisms between two representations of a tensor diagram. Let $R_1(T)$ and $R_2(T)$ be two representations of $T=(V,E)$. A morphism $\Phi:R_1(T)\to R_2(T)$ is a collection of linear maps $\varphi_e\in \Hom(R_1(e),R_2(e))$ for $e\in E$ such that the following diagram commutes for every $v\in V$:

\[
\xymatrix{
\displaystyle{\ott_{i\in\rhu{N}(v)}{R_1(i)}} \ar[r]^{R_1(v)} \ar[d]_{\varphi_{\rhu{N}(v)}} & \displaystyle{\ott_{j\in\lhu{N}(v)}{R_1(j)}} \ar[d]^{\varphi_{\lhu{N}(v)}}\\
\displaystyle{\ott_{i\in\rhu{N}(v)}{R_2(i)}} \ar[r]^{R_2(v)} & \displaystyle{\ott_{j\in\lhu{N}(v)}{R_2(j)}}
}
\]
where $\varphi_{\lhu{N}(v)}:=\ott_{j\in\lhu{N}(v)}{\varphi_j}$ and $\varphi_{\rhu{N}(v)}$ similarly.

We note that it is entirely possible for some vertex $v$ to have the property that either $\rhu{N}(v)$ or $\lhu{N}(v)$ is $\emptyset$. In that case, we recall that the empty tensor product is defined to be $\C$.

We see immediately that $\Hom(R_1(T),R_2(T))$ forms a vector space and that composition of morphisms $\varphi:R_1(T)\to R_2(T)$ and $\psi:R_2(T)\to R_3(T)$ is bilinear, given by a collection of maps $\psi_e\varphi_e$. As such this category is enriched over the category of abelian groups.

It is clear that a morphism is a monomorphism if and only if each $\varphi_e$ is injective and likewise an epimorphism if and only if every $\varphi_e$ is a surjection. Let us consider a monomorphism $\varphi:R_1(T)\into R_2(T)$ given by a collection of maps $\varphi_e$, $e\in E$. Consider a map $\varphi_e:R_1(e)\into R_2(e)$; let $R_3(e)$ denote the cokernel of this map and let $\psi_e$ be the projection map $R_2(e)\onto R_3(e)$. Then for every $v\in V$, define $R_3(v)$ to be $$R_3(v):=\bigg(\ott_{j\in\rhu{N}(v)}{\psi_j}\bigg) R_2(v) \bigg(\ott_{i\in\lhu{N}(v)}{\psi_i}\bigg).$$
So the collection of maps $\psi_e$ define a morphism $\psi:R_2(T)\to R_3(T)$. Lastly, let us denote by $0$ the representation of $T$ given by assigning $\{0\}$ to every wire and the zero map to every vertex. Then we see that 
$$0\to R_1(T)\xrightarrow{\varphi} R_2(T)\xrightarrow{\psi} R_3(T)\to 0$$ is a short exact sequence. So every monomorphism is normal. The representation $R_3(T)$ is the cokernel of $\varphi$. Note that our construction of $R_3(T)$ did not depend on $\varphi$ being a monomorphism, so every map has a cokernel. We also need to show that every epimorphism is normal and every map has a kernel. However, this will follow by describing dual representations that reverse arrows in short exact sequences. We do this now.

Given a representation $R(T)$, we define its dual representation $R^*(T)$ as follows (this necessarily involves a choice of basis). For every $e\in E$, we define $R^*(e):=R(e)^*$ and for every $v$, $R^*(v):=R(v)^T$.

Now suppose that we have a morphism $\varphi:R_1(T)\to R_2(T)$ defined by a collection of maps $\varphi_e$. Then we have dual morphism $\varphi^T:R_2(T)^* \to R_1(T)^*$ given by the maps $\varphi_e^T$. Furthermore, it is easy to check that $\varphi$ is a monomorphism if and only if $\varphi^T$ is an epimorphism and vice versa.

Now suppose we have an epimorphism $\varphi:R_1(T)\onto R_2(T)$. We look at the monomorphism $\varphi^T:R_2(T)^*\into R_1(T)^*$. We know that we can find a representation $R_3(T)$ and morphism $\psi$ such that the sequence
$$0\to R_2(T)^* \xrightarrow{\varphi^T} R_1(T)^* \xrightarrow{\psi} R_3(T)\to 0$$ is exact since every monomorphism is normal. But then we get the following short exact sequence
$$0\to R_3(T)^* \xrightarrow{\psi^T} R_1(T) \xrightarrow{\varphi} R_1(T)\to 0$$ which shows that the epimorphism $\varphi$ is normal. Furthermore, given a morphism $\varphi: R_1(T)\to R_2(T)$, the kernel of this map is isomorphic to the cokernel of the map $\varphi^T:R_2^*(T)\to R_1^*(T)$. So every morphism has a kernel as well. We have now proved the following proposition.

\begin{proposition}
For every tensor diagram $T$, every monomorphism and epimorphism in the category $\mcR(T)$ is normal. Furthermore, this every morphism has a kernel and cokernel
\end{proposition}
We now wish to describe how to take the direct sum of two representations of $T$. This will be a map $\oplus:\mcR_{\bfd}(T)\times \mcR_{\bfd'}(T)\to \mcR_{\bfd+\bfd'}(T)$, where $\bfd+\bfd'$ is coordinate-wise addition. 

Let $R(T)$ be a representation of $T$. For $i\in E$,  let $R(i)$ denote the vector space associated to $i$ and for $v\in V$, let $R(v)$ denote the matrix associated to $v$. Now consider two representations $R_1(T)$ and $R_2(T)$ of $T$. For $i\in E$, we define $R_1\oplus R_2(i):=R_1(i)\oplus R_2(i)$. To define $R_1\oplus R_2(v)$ is slightly trickier, however. We recall the tensor direct sum $\oplus_t$.

\begin{defn}
 Let $T_1\in V=\ott_{i=1}^{n}{V_i}$ and $T_2\in W=\ott_{i=1}^{n}{W_i}$. Then we note that there are natural projections of $\ott_{i=1}^{n}{(V_i\op W_i)}$ onto $V$ and $W$,  denote them $p_V$ and $p_W$, respectively. Then $T_1\op_t T_2$ is the unique tensor $T\in \ott_{i=1}^{n}{(V_i\op W_i)}$ such that $p_V(T)=T_1$ and $p_W(T)=T_2$.
\end{defn}
\noindent We now define $R_1\oplus R_2(v):=R_1(v)\op_t R_2(v)$. We note that if $$R_1(v)\in\Hom\bigg(\ott_{i\in\rhu{N}(v)}{k^{d_i}},\ott_{j\in\lhu{N}(v)}{k^{d_j}}\bigg)\cong \ott_{i\in\rhu{N}(v)}{(k^{d_i})^*}\ot \ott_{j\in\lhu{N}(v)}{k^{d_j}},$$
$$R_2(v)\in\Hom\bigg(\ott_{i\in\rhu{N}(v)}{k^{d'_i}},\ott_{j\in\lhu{N}(v)}{k^{d'_j}}\bigg)\cong \ott_{i\in\rhu{N}(v)}{(k^{d'_i})^*}\ot \ott_{j\in\lhu{N}(v)}{k^{d'_j}},$$ and so $R_1(v)\op_t R_2(v)$ is in the space $$\ott_{i\in\rhu{N}(v)}{(k^{d_i}\op k^{d'_i})^*}\ot\ott_{j\in\lhu{N}(v)}{(k^{d_j}\op k^{d'_j})}\cong \Hom\bigg(\ott_{i\in\rhu{N}(v)}{k^{d_i+d'_i}},\ott_{j\in\lhu{N}(v)}{k^{d_j+d'_j}}\bigg)$$ which is the correct space given how we defined the direct sum of $R_1(T)$ and $R_2(T)$ on the wires.

We see that the direct sum of a finite number of representations is both a product and coproduct in this category. Thus it makes perfect sense to talk about irreducible, simple, and semi-simple representations. This completes the proof that $\mcR(T)$ is an abelian category.

\begin{theorem}\label{thm:abeliancategory}
 Given a tensor diagram $T$, $\mcR(T)$ is an abelian category.
\end{theorem}

Actually, this category also has the structure of a closed compact monoidal category as well. While we do not prove this rigorously, it is easy to see. To define $R_1\ot R_2(T)$, define $R_1\ot R_2(e):=R_1(e)\ot R_2(e)$ and $R_1\ot R_2(v):=R_1(v)\ot R_2(v)$ for all $e$ and $v$. The monoidal unit is given by the representation $\unit(T)$, defined as $\unit(e)=\C$ for all $e$ and $\unit(v)=1$ for all $v$. There is a clear isomorphism $R_1\ot R_2(T)$ and $R_2\ot R_1(T)$ that defines a symmetric braiding. Lastly, we have already demonstrated how to construct dual objects and morphisms.

\section{Forbidden Subdiagrams}\label{sec:forbsuddiag}

In the representation theory of any object, the most important task is to classify the indecomposable representations. In some cases, one can hope for a finite number of indecomposable representations, or that every indecomposable representation belongs to one of a finite number of one parameter families. As one encounters two (and higher) parameter families of indecomposable representations, classification becomes more difficult.

On the flip side, there are objects whose representation theory is prohibitively complex. This is because their category of modules contains the category of modules of every finite dimensional algebra as a full subcategory. An unfortunate motif in representation theory is that once an object has a two-parameter family of indecomposable representations, it become intractable in this way. We make these notions precise and show that for tensor diagrams, such a trichotomy holds.

\begin{defn}
 Given a tensor diagram $T$, we say that $T$ is
 \begin{itemize}
 \item \emph{finite} if there are finitely many indecomposable representations in $\mcR_\bfd(T)$.
  \item \emph{tame} if for every dimension vector $\bfd$, all but finitely many of  the indecomposable representations in $\mcR_\bfd(T)$ belong to one of a finite number of one-parameter families.
  \item \emph{wild} if for every finite dimensional algebra $A$, $A$-Mod is equivalent to a full subcategory of $\mcR(T)$.
 \end{itemize}
\end{defn}

 We note that in our definition of wild, it is true but not obvious that for every $n$, there is an $n$-parameter family of indecomposables. This follows from the first work on tame-wild dichotomy theorems by Drozd \cite{drozd1980tame}, whose work implied these dichotomies for quivers and finite dimensional algebras over algebraically closed fields. Gabriel famously classified the quivers of finite type, given by simply laced Dynkin diagrams \cite{gabriel1972unzerlegbare}.

Our goal in this section is to find a set of forbidden subdiagrams in the sense that any tensor diagram containing them will be wild. The first forbidden subdiagram is given by the following theorem. We shall make a few modifications which will give us the other forbidden subdiagrams.

\begin{theorem}[\cite{belitskii2003complexity}]\label{thm:trivalentwild}
 The problem of classifying the orbits of $\GL(m)\times \GL(n)\times\GL(q)$ acting on $\C^m\ot \C^n\ot \C^q$ is wild.
\end{theorem}

Theorem \ref{thm:trivalentwild} states that the tensor diagram with a single vertex and three dangling wires is wild. From this, one would expect that any tensor diagram with a vertex of degree at least three is wild. Indeed, this will be the case, but will require some work to show. There also exist other forbidden subdiagrams that we need to prove are wild. Lastly we need to show that tensor diagrams with degree at most two are tame, and determine which of these are of finite type.

We first give a list of forbidden subdiagrams that, if in a tensor diagram, will force it to be wild. The list is given in Figure \ref{fig:forbiddensubdiagrams}. Theorem \ref{thm:trivalentwild} states that the open claw is wild. We shall need to prove that the needle and the figure eight are also wild. In fact, the underlying semi-graph of a tensor network is all that is needed to determine if it is of finite, tame, or wild type. So just as in the case of quivers, these properties are orientation independent. We prove this now.
\tikzstyle{block} = [rectangle, 
    text width=1em, anchor=west, minimum height=1em]
\begin{figure}
 \centering
 \subfigure[The open claw.]{
 \begin{tikzpicture}
  \draw[fill=black] (2,0) circle(.05);
  \draw[->] (2,0) -- (2.25,.375);
  \draw (2.25,.375) -- (2.5,.75);
  \draw[->] (2,0) -- (2.25,0);
  \draw (2.25,0) -- (2.5,0);
  \draw[->] (2,0) -- (2.25,-.375);
  \draw (2.25,-.375) -- (2.5,-.75);
  \draw[draw=none] (1,0) rectangle (4,1);
 \end{tikzpicture}
 }\qquad
 \subfigure[The needle.]{
 \begin{tikzpicture}
  \draw[fill=black] (0,0) circle(.05);
  \draw[->] (0,0) arc(270:90:.25);
  \draw (0,0) arc(-90:90:.25);
  \draw[->] (0,0) -- (0,-.25);
  \draw (0,-.25)--(0,-.5);
  \draw[draw=none] (-1,0) rectangle (1,1);
 \end{tikzpicture}
 }\qquad\qquad
 \subfigure[The figure eight.]{
 \begin{tikzpicture}
  \draw[fill=black] (0,0) circle(.05);
  \draw[->] (0,0) arc(270:90:.25);
  \draw (0,0) arc(-90:90:.25);
  \draw[->] (0,0) arc(90:-90:.25);
  \draw (0,0) arc(90:270:.25);
  \draw[draw=none] (-1.25,0) rectangle (1.25,1);
 \end{tikzpicture}
 }\caption{Forbidden subdiagrams of tame tensor diagrams. All other orientations of these tensor diagrams given forbidden subdiagrams.}\label{fig:forbiddensubdiagrams}
\end{figure}
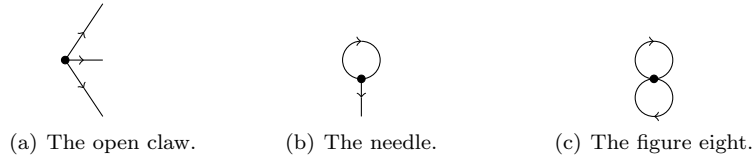

\begin{lemma}\label{lem:orientationindependence}
 Let $T$ and $T'$ be two tensor diagrams with the same underlying semi-graph but with the orientation on a single wire reversed. Then $\mcR(T)\cong \mcR(T')$.
\end{lemma}
\begin{proof}
 We construct a map $\varphi:\mcR(T)\to\mcR(T')$ sending $R(T)\in\mcR(T)$ to $R_\varphi(T')\in\mcR(T')$. Let $T=(V,E)$ and $T'=(V,E')$ where $E'$ is obtained from $E$ by reversing the orientation on a single wire, call it $w$. Then for $e\in E'\setminus\{w\}$, we define $R_\varphi(e)=R(e)$. For all $v\in V$ not incident to $w$, we define $R_\varphi(v)=R(v)$. Now we choose a basis for $R(w)$, which induces an isomorphism $\psi:R(w)\cong R(w)^*$. Suppose $w\in\rhu{N}(v)$ for some $v\in V$. Then $\psi$ induces an isomorphism
 $$\tilde{\psi}:R(w)^*\ot\ott_{i\in\rhu{N}(v)\setminus\{w\}}{R(i)^*}\ot\ott_{j\in\lhu{N}(v)}{R(j)}\cong $$
 $$\ott_{i\in\rhu{N}(v)\setminus\{w\}}{R(i)^*}\ot\ott_{j\in\lhu{N}(v)}{R(j)}\ot R(w).$$ We then define $R_\varphi(v)=\tilde{\psi}(R(v))$. We construct a similar isomorphism $\hat{\psi}$ if $w\in\lhu{N}(v)$ and then define $R_\varphi(v)=\hat{\psi}(R(v))$. The map $\varphi$ clearly gives an isomorphism $\mcR(T)\cong\mcR(T')$.
\end{proof}

\subsection{Wildness of the needle and figure eight}

We first focus on showing the wildness of the needle given in Figure \ref{fig:forbiddensubdiagrams} (b) noting that the tensor diagram given by any other orientation on the needle will also be wild by Lemma \ref{lem:orientationindependence}. We denote the needle by $N=(V,E)$ where $V=\{v\}$ and $E=\{e_1,e_2\}$ where $e_1$ is the loop and $e_2$ is the dangling wire. The dimension vector $\bfd=(d_1,d_2)$ for a representation $R(N)$ means that $\dim(R(e_i))=d_i$ for $i=1,2$.

A representation $R(N)$ with dimension vector $(d_1,d_2)$ is a single tensor $\sum_{i=1}^{d_2}{M_i\ot u_i}$ for $M_i\in \End(\C^{d_1})$ and $u_i\in \C^{d_2}$. The group $\GL(N)$ acts on $R(N)$ as follows: for $g=(g_1,g_2)\in\GL(d_1)\times \GL(d_2)$, $$g.\sum_{i=1}^{d_2}{M_i\ot u_i}=\sum_{i=1}^{d_2}{g_1M_ig_1^{-1}\ot g_2u_i}.$$

\begin{lemma}\label{lem:needlematrixprob}
 The problem of determining if two representations $R_1(N),R_2(N)\in\mcR_\bfd(N)$, $\bfd=(d_1,d_2)$, are isomorphic is equivalent to determining if two $d_2$-tuples of $d_1\times d_1$ matrices $(A_1,\dots,A_{d_2})\sim(B_1,\dots,B_{d_2})$ under the following equivalence relation:
 \begin{enumerate}[(a)]
  \item For $g\in \GL(d_1)$, $(M_1,\dots,M_{d_2})\sim (gM_1g^{-1},\dots,gM_{d_2}g^{-1})$.
  \item For $g=\{g_{jk}\}\in \GL(d_2)$, $(M_1,\dots,M_{d_2})\sim (\sum_{j=1}^{d_2}{M_{j}g_{j1}},\dots,\sum_{j=1}^{d_2}{M_{j}g_{jd_2}})$.
 \end{enumerate}
\end{lemma}
\begin{proof}
 Given $R(N)=\sum_{i=1}^{d_2}{M_i\ot u_i}$, we associate it to the tuple $(M_1,\dots,M_{d_2})$. Call this map $\psi$. The action $(g,1)\in\GL(N)$ takes $\sum_{i=1}^{d_2}{M_i\ot u_i}$ to $\sum_{i=1}^{d_2}{gM_ig^{-1}\ot u_i}$ which $\psi$ maps to $(gM_1g^{-1},\dots,gM_{d_2}g^{-1})$. The action $(1,g)\in\GL(N)$, with $g=\{g_{jk}\}$, takes $\sum_{i=1}^{d_2}{M_i\ot u_i}$ to $\sum_{i=1}^{d_2}{M_i\ot gu_i}$ which $\psi$ maps to $$(\sum_{j=1}^{d_2}{M_{j}g_{j1}},\dots,\sum_{j=1}^{d_2}{M_{j}g_{jd_2}}).$$
\end{proof}

We slightly modify the proof of Theorem \ref{thm:trivalentwild} given in \cite{belitskii2003complexity} to prove that for sufficiently large $d_1$, the problem of simultaneous similarity of two $n\times n$ matrices can be embedded into the problem of determining isomorphism classes of $R(N)$ with dimension vector $(d_1,2)$. We can do this for every $n$, and so determining the isomorphism classes of $R(N)$ is as difficult as determining simultaneous similarity of two $n\times n$ matrices, which is a wild problem.

\begin{theorem}\label{thm:needlewild}
 The needle is wild.
\end{theorem}
\begin{proof}
 Let us first start with a two pairs $n\times n$ matrices $(A_1,B_1)$ and $(A_2,B_2)$. We want to find two representations $R_1(N)$ and $R_2(N)$ that are isomorphic if and only if $(A_1,A_2)$ is equivalent to $(B_1,B_2)$ under simultaneous similarity. Given a pair of $n\times n$ matrices $(A,B)$, we associate to the the pair of matrices $(Y_1,Y_2(A,B)):=(X_1\op C_1,X_2\op C_2(A,B))$ where
 \begin{equation*}
  X_1=\begin{pmatrix}
   I_n&0\\
   0&0
  \end{pmatrix},\quad
  X_2=\begin{pmatrix}
0&0\\
0&I_n
\end{pmatrix},
 \end{equation*}
\begin{equation*}
C_1= \begin{pmatrix}
  0& & &\\
  I_n&0& &\\
  0&I_n&0&\\
  0&0&I_n&0
 \end{pmatrix},\quad
 C_2(A,B)=\begin{pmatrix}
  0& & & \\
  0&0& &\\
  A&0&0& \\
  0&B&0&0
 \end{pmatrix}
\end{equation*}

Now suppose that $(Y_1,Y_2(A_1,B_1))\sim (Y_1,Y_2(A_2,B_2))$ under the equivalence defined in Lemma \ref{lem:needlematrixprob}. Then there exists a matrix in $g=\{g_{jk}\}\in\GL(2)$ such that 
$$(Y_1g_{11}+Y_2(A_1,B_1)g_{21},Y_1g_{12}+Y_2(A_1,B_1)g_{22})$$ is simultaneously similar to $(Y_1,Y_2(A_2,B_2))$. Note that this implies that $\rank(Y_1g_{11}+Y_2(A_1,B_1)g_{21})$ is equal to $\rank(Y_1)$ and similarly $\rank(Y_1g_{12}+Y_2(A_1,B_1)g_{22})=\rank(Y_2(A_2,B_2))$. But $\rank(Y_1g_{11}+Y_2(A_1,B_1)g_{21})=5n>Y_1$ if both $g_{11},g_{21}\ne 0$. Furthermore $\rank(Y_1)=4n>3n\ge \rank(Y_2(A_1,B_1))$, so $g_{21}=0$ and $g_{11}\ne 0$.

Next we note that  $\rank(Y_1g_{12}+Y_2(A_1,B_1)g_{22})>\rank(Y_2(A_2,B_2))$ if both $g_{12},g_{22}\ne 0$. If $g_{22}=0$, then $g$ was singular, so we conclude that $g_{12}=0$ and $g_{22}\ne 0$. So we see that $g_{11}Y_1$ is similar to $Y_1$, implying that $g_{11}=1$, and that $g_{22}Y_2(A_1,B_1)$ is similar to $Y_2(A_2,B_2)$, implying that $g_{22}=1$.

So we have that $(Y_1,Y_2(A_1,B_1))$ must be simultaneously similar to $(Y_1,Y_2(A_2,B_2))$. This implies that $(C_1,C_2(A_1,B_1))$ and $(C_1,C_2(A_2,B_2))$ are simultaneously similar. Let $g\in\GL(4n)$ be such that $gC_1=C_1g$ and $gC_2(A_1,B_1)=C_2(A_2,B_2)g$. The first equation implies that 
\begin{equation*}
 g=\begin{pmatrix}
    g_1&  & &\\
    g_2& g_1& &\\
    g_3&g_2&g_1&\\
    g_4&g_3&g_2&g_1
   \end{pmatrix}
\end{equation*}
The second equation implies that $A_1g_1=g_1A_2$ and $B_1g_1=g_1B_2$.
\end{proof}

We now look to the figure eight. The proof that this diagram is wild is also very similar to the needle and the open claw. Let $B$ denote the figure eight. We only consider dimension vectors of the form $(d_1,2)$ again and show that there is an embedding of simultaneous similarity of $n\times n$ matrices into the representations of $B$. Given a representation $R(B)$ with dimension vector $(d_1,d_2)$, we note that it is a tensor of the form $\sum_{i=1}^{d_2}{M_i\ot E_i}$ where $M_i\in\End(\C^{d_1})$ and $E_i\in\End(\C^{d_2})$ are the elementary basis matrices.

\begin{lemma}\label{lem:8matrixproblem}
 The problem of determining if two representations $R_1(B),R_2(B)\in\mcR_\bfd(N)$, $\bfd=(d_1,d_2)$, are isomorphic is equivalent to determining if two $d_2$-tuples of $d_1\times d_1$ matrices $(A_1,\dots,A_{d_2})\sim(B_1,\dots,B_{d_2})$ under the following equivalence relation:
 \begin{enumerate}[(a)]
 \item For $g\in \GL(d_1)$, $(M_1,\dots,M_{d_2^2})\sim (gM_1g^{-1},\dots,gM_{d_2^2}g^{-1})$.
 \item For $h=g\ot (g^{-1})^T=\{h_{jk}\}$, $g\in \GL(d_2)$, $$(M_1,\dots,M_{d_2^2})\sim (\sum_{j=1}^{d_2^2}{M_{j}h_{j1}},\dots,\sum_{j=1}^{d_2^2}{M_{j}h_{jd_2^2}}).$$
 \end{enumerate}
\end{lemma}
\begin{proof}
 We construct a map $\psi$ that takes $\sum_{i=1}^{d_2}{M_i\ot E_i}$ to the tuple $(M_1,\dots, M_{d_2^2})$. Part (a) follows the exact same lines as in Lemma \ref{lem:needlematrixprob}. For part (b), we note that $\End(\C^{d^2})\cong \C^{d_2}\ot \C^{d_2}$ and that the induced action of $\GL(d_2)$ on this space is by $g\ot (g^{-1})^T$. Then part (b) follows along the same lines as in Lemma \ref{lem:needlematrixprob}.
\end{proof}

\begin{proposition}\label{prop:8wild}
 The figure eight is wild.
\end{proposition}
\begin{proof}
 Given two pairs of $n\times n$ matrices $(A_1,B_1)$ and $(A_2,B_2)$, we construct the four tuples $(Y_1,Y_2(A_1,B_1),0,0)$ and $(Y_1,Y_2(A_2,B_2),0,0)$, where $Y_1$ and $Y_2(A,B)$ are defined in exactly the same way as in Theorem \ref{thm:needlewild}. Following the same lines as in the proof of Theorem \ref{thm:needlewild}, we can conclude that $(Y_1,Y_2(A_1,B_1),0,0)$ and $(Y_1,Y_2(A_2,B_2),0,0)$ are related by the equivalence relation defined in Lemma \ref{lem:8matrixproblem} if and only if $(Y_1,Y_2(A_1,B_1))$ and $(Y_1,Y_2(A_2,B_2))$ are simultaneously similar. From here, we conclude in the same fashion as in Theorem \ref{thm:needlewild} that this happens if and only if $(A_1,B_1)$ and $(A_2,B_2)$ are simultaneously similar.
\end{proof}

\section{Tensor diagrams containing wild subdiagrams are wild}\label{sec:inclusion}
Our next goal is to show that if $T$ is a tensor diagram and $S$ a subdiagram, then under certain conditions $\mcR(S)$ is a equivalent to a full subcategory of $\mcR(T)$. For this, we first reduce the problem to the case where $S$ is an induced subdiagram of $T$.

\begin{defn}
 Given a tensor diagram $T=(V,E)$, a vertex $v\in V$, and a partitioning of the wires incident to $v$, $W_1\sqcup W_2=N(v)$, a \emph{splitting} of $T$ with respect to $(v,W_1,W_2)$ is the tensor diagram formed by
 \begin{enumerate}
  \item Replacing $v$ with two distinct vertices $v_1$ and $v_2$
  \item The wires $W_1$ become the wires incident to $v_1$ and likewise $W_2$ are the wires incident to $v_2$.
  \item A wire from $v_1$ to $v_2$ is added.
 \end{enumerate}
\end{defn}

The operation of splitting is the reverse operation of contraction along an edge. Given a tensor diagram $T=(V,E)$, a wire $e\in E$, and $c\in\N$, we denote by $\mcR(T)^{e=c}$ the subcategory of $\mcR(T)$ of all representations of $T$ where the dimension of the vector space associated to $e$ is $c$. This category is not abelian as it is not closed under direct sums. If $c=1$, however, it forms a closed compact monoidal category as $\C\ot \C\cong \C$. More generally, given a set $r=\{(e_i,c_i)|\; e_i\in E,\; c_i\in\N\}$, we can consider the restricted category $\mcR(T)^r$ given by restriction to representations where the vector space associated to $e_i$ has dimension $c_i$. Again, if all $c_i=1$ in $r$, $\mcR(T)^r$ forms a closed compact monoidal category.

\begin{lemma}\label{lem:splitting}
 Given a tensor diagram $T=(V,E)$, let $T'$ be the tensor diagram formed by a splitting of $T$ with respect to $(v,W_1,W_2)$ and $e$ be the added wire between $v_1$ and $v_2$. Then $\mcR(T')^{e=1}$ is equivalent to a full subcategory of $\mcR(T)$.
\end{lemma}
\begin{proof}
  Given a representation $R(T')\in \mcR(T')^{e=1}$, we define a map $\phi:R(T')\mapsto R_\varphi(T)\in \mcR(T)$ as follows. Note that $T'=(V',E')$ where $V'=(V\setminus \{v\})\cup\{v_1,v_2\}$ and $E'=E\cup\{e\}$. For $u\in V'$, $u\ne v_1,v_2$ we define $R_\varphi(u):=R(u)$, and for every $w\in E'$, $w\ne e$, $R_\varphi(e):=R(w)$. Then we define $R_\varphi(v):=R(v_1)\ot R(v_2)$. 
  
  We note that the non-injectivity of the map $\varphi$ comes from the fact that $\alpha R(v_1)\ot \alpha^{-1} R(v_2)=R(v_1)\ot R(v_2)$. Let $R_1(T')$ and $R_2(T')$ be two representations in $\mcR(T')^{e=1}$. This implies that for all $u\in V'$, $u\ne v_1,v_2$, $R_1(u)=R_2(u)$. Furthermore, there exists a non-zero $\alpha\in\C$ such that $R_1(v_1)=\alpha R_2(v_1)$ and $R_1(v_2)=\alpha^{-1} R_2(v_2)$. But this implies that $R_1(T')\cong R_2(T')$ by applying the change of basis $\alpha$ on the wire $e$. Therefore, the map $\varphi$ is injective on isomorphism classes of representations in $\mcR(T')^{e=1}$, showing that this category is equivalent to a subcategory of $\mcR(T)$. The fullness of the subcategory is obvious.
\end{proof}

We note that Lemma \ref{lem:splitting} can be applied inductively to show that $\mcR(T')^r$ is equivalent to a full subcategory of $\mcR(T)$, where $T'$ is formed from $T$ by a sequence of splittings and $r=\{(e_i,1)\}$, where $e_i$ is the edge added by the $i^{th}$ splitting.

Let $T=(V,E)$ be a tensor diagram and $S=(U,F)$ be a subdiagram. We define the following sequence of splittings of $T$:
\begin{enumerate}
 \item For every $v\in U$, partition the wires incident to $v$ into $W_1\sqcup W_2$, where $W_1$ is the set of the wires whose other endpoint lies in $U$. Split $T$ with respect to $(v, W_1,W_2)$. After all such splittings we get a new diagram $T_1$. We now define $U_1$ to be all the vertices $v_1$ formed at each splitting of $v\in U$. Note that $S$ is a subdiagram of the subdiagram induced by $U_1$ in $T_1$.
 \item For every $v_1\in U_1$, we partition the wires incident to $v_1$ into $W_1\sqcup W_2$ where $W_1$ consists of the wire from $v_1\to v_2$ formed by the splitting in the previous step, as well as all wires in $F$. Then we split $T_1$ with respect to $(v_1,W_1,W_2)$. We get two new vertices $v_{11}$ and $v_{12}$. We let $U_S$ be the set of all vertices of type $v_{11}$. After all such splittings, we get a new diagram $T_S$. 
\end{enumerate}

We call the above construction \emph{isolating $S$ in $T$}. We note that the subdiagram of $T_S$ induced by $U_S$, $\tilde{S}$, is isomorphic to $S$. Let $I_S$ be the wires added by the splittings transforming $T$ to $T_S$. Let $r_S=\{(e, 1)|\; e\in I_S\}$, then by Lemma \ref{lem:splitting}, $\mcR(T_S)^{r_S}$ is equivalent to a full subcategory of $\mcR(T)$.

We now want to show that $\mcR(\tilde{S})$ is equivalent to a full subcategory of $\mcR(T_S)^{r_S}$. Since $\tilde{S}$ is isomorphic to $S$, this will imply that $\mcR(S)$ is equivalent to a full subcategory of $\mcR(T)$. Since $\tilde{S}$ is an induced subdiagram of $T_S$, it suffices to show that if $T$ is a tensor diagram and $S$ an induced subdiagram then $\mcR(S)$ is equivalent to a full subcategory of $\mcR(T)$. To show this, we will first need to consider abelian flows on tensor diagrams.  We first recall basic definitions.

\begin{defn}
 Given a tensor diagram $T=(V,E)$ and a subset $U\subseteq V$, let $T[U]=(U,F)$ be the induced subdiagram. Let $\overline{U}:=V\setminus U$ and $\overline{F}:=E\setminus F$. A \emph{partial $\C^\times$-flow}  on $T$ with respect to $U\subseteq V$ is a map $f:\overline{F}\to\C^\times$ such that for every $v\in \overline{U}$, $$\prod_{e\in\rhu{N}(v)\cap \overline{F}}{f(e)^{-1}}\times\prod_{e\in\lhu{N}(v)\cap \overline{F}}{f(e)}=1$$ for some $U\subseteq V$. If $U=V$, then we simply call $f$ a $\C^\times$-flow on $T$. 
\end{defn}

\begin{defn}
 Let $T=(V,E)$ be a tensor diagram. Suppose $T$ is a directed graph and let $\tilde{T}=(\tilde{V},\tilde{E})$ be the graph formed from $T$ by ignoring the orientation. Then we form a new graph $\hat{T}$ as follows.
 \begin{enumerate}[(a)]
  \item For every pair of adjacent vertices $u,v\in\tilde{V}$, let us merge all wires between them into a single wire.
  \item Remove all loops from $\tilde{T}$.
 \end{enumerate}
The resulting graph $\hat{T}$ is simple. If it is a tree, we call $T$ a \emph{multi-tree}.
\end{defn}

We first consider the case when $T$ is closed. Suppose we have a closed tensor diagram $T=(V,E)$ and a partial $\C^\times$-flow $f$ defined with respect to $U\subseteq V$. Then $T$ minus the induced subgraph $T[U]$ is a tensor diagram whose dangling wires are precisely the wires that represent a cut set between $U$ and $\overline{U}$. We denote this set of wires by $D$. We may assume without loss of generality that every wire in $D$ is oriented towards $U$ by Lemma \ref{lem:orientationindependence}. The partial flow condition then implies that $\prod_{e\in D}{f(e)}=1$.

\begin{lemma}\label{lem:flowextension}
 Let $T=(E,V)$ be  a closed tensor diagram and $U\subseteq V$. Let $f$ be a partial $\C^\times$-flow defined with respect to $U\subseteq V$, then $f$ can be extended to a $\C^\times$-flow on $T$.
\end{lemma}
\begin{proof}
 As before, let $D$ be the cut set separating $T[U]=(U,F)$ from the complementary subgraph in $T$. We may assume that $T[U]$ has no loops as they contribute trivially to flows. We assume without loss of generality that every $e\in D$ is oriented towards $U$ using Lemma \ref{lem:orientationindependence}. Let us first assume that $\hat{T}[U]$  is a star. We may assume that all wires in $T[U]$ are oriented towards the central vertex $c$ by Lemma \ref{lem:orientationindependence}. Let $v\in U\setminus\{c\}$ be incident to the wires $e_{v}^1,\dots,e_{v}^k\in F$. Let $W$ be the wires incident to $v$ in $D$. Then to we define the extension of $f$, $\hat{f}$, to send every $e_{v}^i$ to one of the $k^{th}$ roots of $\prod_{w\in W}{f(w)}$. We need to show this defines $\C^\times$-flow on $T$. First of all, it is clear that the flow condition is satisfied for all of the vertices of $T[U]$ except for $c$. However, the flow condition is satisfied for $c$ by the fact that $\prod_{e\in D}{f(e)}=1$. So $\hat{f}$ is a $\C^\times$-flow on $T$.
 
 Now let $\hat{T}[U]$ be a tree. Let $L$ be the leaves of $\hat{T}[U]$. Let $e^{1}_v,\dots,e^k_v$ be a wires incident to $v\in L$, which we may assume are directed away from the leaf by Lemma \ref{lem:orientationindependence}. Let $W$ be the wires incident to $v$ in $D$, which we may assume are oriented towards $v$, again by Lemma \ref{lem:orientationindependence}. Then for all $i\in[k]$, we define $\hat{f}(e^{i}_v)$ to be one of the $k^{th}$ roots of $\prod_{w\in W}{f(w)}$. We have now extended the partial $\C^\times$-flow and we must extend it again to a submulti-tree of $T[U]$. We repeat this procedure until $\hat{T}[U]$ is a star and then we are done.

 We now consider the case where $S=(U,F)$ is an arbitrary subdiagram of $T$. We first choose a spanning tree of $S$, $G=(U,F')$. We may extend $f$ to a function $\hat{f}$ onto the edges of $G$, viewed as a subdiagram of $T[U]$, using the argument above. Then for every $e\in F\setminus F'$, we define $\hat{f}(e)=1$. This clearly defines a $\C^\times$-flow on $T$.
\end{proof}

Let $(U,F)$ be the induced subgraph of $T_S=(V,E)$, which by an abuse of notation we call $S$. Let $r:=\{(e,1)|\; e\in E\setminus F\}$ and note that $r_S\subseteq r$. This implies that $\mcR(T_S)^r$ is a full subcategory of $\mcR(T_S)^{r_S}$. We now consider a map $\varphi:\mcR(S)\to \mcR(T_S)^r$. Given $R(S)\in \mcR(S)$, the map $\varphi$ takes $R(S)$ to a representation $R_\varphi(T_S)$ in $\mcR(T_S)^r$ as follows. For each $e\in F$, $R_{\varphi}(e)=R(e)$ and for each $v\in U$, $R_\varphi(v)=R(v)$ (possibly using an isomorphism $\ott{V_i}\cong\ott{V_i}\ot \C^{\ot n}\ot(\C^*)^{\ot m}$ sending $\ott{v_i}\mapsto \ott{v_i}\ot 1^{\ot n}\ot (1^*)^{\ot m}$).For each $e\in\overline{F}$, $R_\varphi(e)=\C$ and for each $v\in\overline{U}$, $$R_\varphi(v)=\bigg(\ott_{e\in\rhu{N}(v)}{1^*}\bigg)\ot\bigg(\ott_{e\in\lhu{N}(v)}{1}\bigg).$$

\begin{lemma}\label{lem:subdiagram}
 The map $\varphi:\mcR(S)\to \mcR(T_S)^r$ realizes the equivalence of $\mcR(S)$ to a full subcategory of $\mcR(T_S)^r$.
\end{lemma}
\begin{proof}
 We simply need to show that if $R_\varphi(T_S)$ and $R'_\varphi(T_S)$ are isomorphic for $R(S),R'(S)\in \mcR(S)$, then $R(S)\cong R'(S)$. So suppose that $R_\varphi(T_S)$ and $R'_\varphi(T_S)$ are related by an element $G\in\GL(T_S)$. We may assume that the action of $G$ restricted to $S$ is the identity (again by abuse of notation, by $S$ we mean the induced subdiagram of $T_S$, $(U,F)$, isomorphic to $S$). So the action of $G$ on every $v\in \overline{U}$ sends the tensor $$\bigg(\ott_{e\in\rhu{N}(v)}{1^*}\bigg)\ot\bigg(\ott_{e\in\lhu{N}(v)}{1}\bigg)$$ to itself. We see that the action of $G$ describes a partial $\C^\times$-flow on $T$ with respect to $U$. Let $D$ be the set of wires from $T[U]$ to the complementary subdiagram as before. The action of $G$ on the wires in $D$ acts non-trivially on the tensors associated to some of the vertices in $U$. We need to show that there is a representation of $S$, $R''(S)\cong R(S)$ such that the induced action of $G$ on $R''_\varphi(T_S)$ restricted to the 
vertices $U$ acts trivially. By Lemma \ref{lem:flowextension}, we may extend the partial $\C^\times$-flow induced by $G$ to a $\C^\times$-flow on all of $T_S$. This defines an element of $H\in\GL(S)$. Let $R''(S):=H.R(S)$, which is by definition an isomorphic representation of $S$. Since $\GL(S)$ is a subgroup of $\GL(T_S)$ in a natural way, $H$ lifts to a an element $\tilde{H}\in \GL(T_S)$. Furthermore, it commutes with $G$. Then if we look at $\tilde{H}G. R_\varphi(T_S)$, we see that this gives the flow on $T_S$ guaranteed by Lemma \ref{lem:flowextension} and that this group element is in the isotropy group of $\GL(T_S)$ acting on $T_S$. So we see that if $R_\varphi(T_S)$ and $R'_\varphi(T_S)$ are isomorphic by an element $G\in \GL(T_S)$, then $R(S)$ and $R'(S)$ are isomorphic by an element $H\in GL(S)$. 
\end{proof}

With an eye towards Theorem \ref{thm:trivalentwild}, we are particularly interested in closed tensor diagrams with a vertex of degree $\ge 3$. If none of the incident edges of this vertex is a loop, then there exists a splitting such that the resulting graph has a claw as an induced subgraph and there are no dimension restrictions on the dimensions of the wires in the claw. This is obvious if the tensor diagram is simple; in the non-simple scenario, two other situations arise. Figure \ref{fig:splitting} shows both situations and the resulting splittings realizing the claw as an induced subgraph. Solid wires have no dimension restrictions while the dotted edges are the edges added by the splitting and thus can only have representations of $\C$. Lemma \ref{lem:subdiagram} implies that determining the indecomposable representations of a closed tensor diagram with such vertex of degree $\ge 3$, not counting loops, is as hard as determining the indecomposable representations of the claw diagram. The next lemma 
states that this is as hard as determining the indecomposable representations of trivalent tensors.

\begin{figure}
\centering
 \begin{tikzpicture}
  \draw[fill=black] (0,0) circle (.05);
  \draw (0,-.25) node{$v_1$};
  \draw[fill=black] (1,0) circle(.05);
  \draw (1,-.25) node{$v_2$};
  \draw[->] (.1,.1) -- (.5,.1);
  \draw (.5,.1) -- (.9,.1);
  \draw[->] (.1,0) -- (.5,0);
  \draw (.5,0) -- (.9,0);
  \draw[->] (.1,-.1) -- (.5,-.1);
  \draw (.5,-.1) -- (.9,-.1);
  \draw (1.5,0) node{$\mapsto$};
  \draw[fill=black] (2,0) circle(.05);
  \draw (2,-.25) node{$v_1$};
  \draw[->] (2,0) -- (2.25,.375);
  \draw (2.25,.375) -- (2.5,.75);
  \draw[->] (2,0) -- (2.25,0);
  \draw (2.25,0) -- (2.5,0);
  \draw[->] (2,0) -- (2.25,-.375);
  \draw (2.25,-.375) -- (2.5,-.75);
  \draw[dashed,->](2.5,.75) -- (3,.75);
  \draw[dashed] (3,.75) -- (3.5,.75);
  \draw[dashed,->](2.5,0) -- (3,0);
  \draw[dashed] (3,0) -- (3.5,0);
  \draw[dashed,->](2.5,-.75) -- (3,-.75);
  \draw[dashed] (3.5,-.75) -- (3.5,-.75);
  \draw[dashed,->] (3.5,0) -- (3.5,.25);
  \draw[dashed] (3.5,.25) -- (3.5,.5);
  \draw[dashed,->] (3.5,-.5) -- (3.5,-.25);
  \draw[dashed] (3.5,-.25) -- (3.5,0);
  \draw[fill=black] (3.5,.75) circle(.05);
  \draw[fill=black] (3.5,0) circle(.05);
  \draw[fill=black] (3.5,-.75) circle(.05);
  \draw[fill=black] (2.5,.75) circle(.05);
  \draw[fill=black] (2.5,0) circle(.05);
  \draw[fill=black] (2.5,-.75) circle(.05);
  
  \draw[fill=black] (-1,-2) circle (.05);
  \draw (-1,-2.25) node{$v_2$};
  \draw[fill=black] (0,-2) circle(.05);
  \draw (0,-2.25) node{$v_1$};
  \draw[->] (-.9,-1.9) -- (-.5,-1.9);
  \draw (-.5,-1.9) -- (-.1,-1.9);
  \draw[->] (-.9,-2) -- (-.5,-2);
  \draw (-.5,-2) -- (-.1,-2);
  \draw[->] (.1,-2) -- (.5,-2);
  \draw (.5,-2) -- (.9,-2);
  \draw[fill=black] (1,-2) circle (.05);
  \draw (1,-2.25) node{$v_3$};
  \draw (1.5,-2) node{$\mapsto$};
  \draw[fill=black] (2,-1.5) circle(.05);
  \draw[fill=black] (2,-2.5) circle(.05);
  \draw[dashed,->] (2,-2.5) -- (2,-2);
  \draw[dashed] (2,-2) -- (2,-1.5);
  \draw[dashed,->] (2,-1.5) -- (2.5,-1.5);
  \draw[dashed] (2.5,-1.5) -- (3,-1.5);
  \draw[dashed,->] (2,-2.5) -- (2.5,-2.5);
  \draw[dashed] (2.5,-2.5) -- (3,-2.5);
  \draw[fill=black] (3,-1.5) circle(.05);
  \draw[fill=black] (3,-2.5) circle(.05);
  \draw[->] (3,-1.5) -- (3.25,-1.75);
  \draw (3.25,-1.75) -- (3.5,-2);
  \draw[->] (3,-2.5) -- (3.25,-2.25);
  \draw (3.25,-2.25) -- (3.5,-2);
  \draw[fill=black] (3.5,-2) circle(.05);
  \draw (3.6,-2.25) node{$v_1$};
  \draw[->] (3.5,-2) -- (4,-2);
  \draw (4,-2) -- (4.5,-2);
  \draw[fill=black] (4.5,-2) circle(.05);
  \draw (4.5,-2.25) node{$v_3$};
 \end{tikzpicture}
 \caption{Two splittings of degree three vertices showing the claw as an induced subgraph.}\label{fig:splitting}
\end{figure}
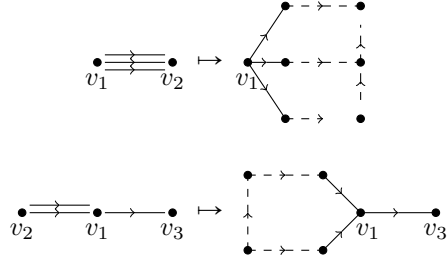

\begin{lemma}\label{lem:openclaw}
 Let $C$ be a claw diagram (with any orientation on the wires) and $C^0$ the open claw diagram. Then $\mcR(C^0)$ is a full subcategory of $\mcR(C)$. Similarly, let $\hat{N}$ be the needle diagram with a vertex added to its dangling wire and $N$ be the needle diagram. Then $\mcR(N)$ is a full subcategory of $\mcR(\hat{N})$.
\end{lemma}
\begin{proof}
 We may assume without loss of generality that all wires are directed away from the leaves by Lemma \ref{lem:orientationindependence}. Now let us construct a map $\varphi:\mcR(C^0)\to \mcR(C)$ sending $R(C^0)\in\mcR(C^0)$ to $R_\varphi(C)\in\mcR(C)$. Let $c$ denote the only vertex in $C^0$. Then $R_\varphi(c):=R(c)$ and for all other vertices $v$ in $C$, define $R_\varphi(v)=0$. It is clear that $R_\varphi(C)\cong R'_\varphi(C)$ if and only if $R(C^0)\cong R'(C^0)$. The fullness of the subcategory is also clear.
 
 Similarly for $R(N)$, we define a map $\varphi:\mcR(N)\to \mcR(\hat{N})$. Given $R(N)\in \mcR(N)$, we define $R_\varphi(\hat{N})\in\mcR(\hat{N})$ as follows: $R_\varphi(e)=R(e)$ for all wires, $R_\varphi(c)=R(c)$ for the single vertex $c$ in $N$, and $R_\varphi(v)=0$, where $v$ is the vertex in $\hat{N}$ that is not in $N$. It is clear that that $R_\varphi(\hat{N})\cong R'_\varphi(\hat{N})$ if and only if $R(N)\cong R'(N)$. The fullness of the subcategory is also clear.
\end{proof}
\begin{corollary}\label{cor:closeddiagramswild}
 Every closed tensor diagram with a vertex of degree $\ge 3$ (not counting loops) is wild. Any closed tensor diagram with a vertex with a loop and a non-empty neighborhood is wild. Any tensor diagram which has a connected component consisting of a single vertex with at least two loops is wild.
\end{corollary}
\begin{proof}
 This follows directly from Lemmas \ref{lem:splitting}, \ref{lem:subdiagram} and \ref{lem:openclaw}, Theorem \ref{thm:trivalentwild}, Theorem \ref{thm:needlewild}, and Proposition \ref{prop:8wild} and the observation that given a disconnected tensor diagram $T=T_1\sqcup T_2$, the categories $\mcR(T_1)$ and $\mcR(T_2)$ are clearly full subcategories of $\mcR(T)$.
\end{proof}

We now turn out attention to tensor diagrams with dangling wires. We first observe that we may assume that a tensor diagram has no wire with no endpoints. The reason is it shares the same vertices as the tensor diagram with this wire removed. As such, these two diagrams have precisely the same category of representations. So we only need look at tensor diagrams where there are wires with one vertex.

Let $T=(V,E)$ be a tensor diagram and let $H\subseteq E$ be the set of dangling wires, i.e. wires with one endpoint. Now let us consider any representation $R(T)$. Let $v\in V$ be incident to some $e\in H$. Then applying either $\alpha I$ or $\alpha^{-1}I$ to the wire $e$ (depending on the orientation of $e$), we get that the representation $R'(T)$ given by $R'(w)=R(w)$ for $w\ne v$ and $R'(v)=\alpha R(v)$ is an isomorphic representation to $R(T)$.

Given $R'(T)$, we may now apply either $\alpha I$ or $\alpha^{-1} I$ to another wire incident to $v$, say $f$, depending on the orientation of $f$, such that $R'(v)$ gets replaced with $\alpha^{-1}R'(v)=R(v)$. If $f\in H$, then we see that multiplying by a scalar on one of the dangling wires incident to $v$ is isomorphic by multiplying by a (potentially different) scalar on a different dangling wire incident to $v$. If $f\notin H$, then let $u$ be its other endpoint, which multiplies $R(y)$ by $\alpha$. Then we get a new representation $R''(T)$, where $R''(u)=\alpha R(u)$ and for $w\ne u$, $R''(w)=R(w)$. We see that $R''(T)$ is isomorphic to $R(T)$. In this way we get the following fact:

\begin{fact}\label{fact:isom}
 Let $T=(V,E)$ be a connected tensor diagram with dangling wires $H$. Given a representation $R(T)$, any representation formed from $R(T)$ by multiplying $R(v)$, $v\in V$, by a non-zero scalar $\alpha$ is an isomorphism of representations. Furthermore, the representation induced by multiplying every $w\in H$ by different scalars is isomorphic to the representation induced by multiplying a single wire in $H$ by a certain scalar.
\end{fact}

\begin{theorem}\label{thm:wildcases}
 If a tensor diagram $T=(V,E)$ with dangling wires $H\subseteq E$ contains the open claw, needle, or figure eight as a subdiagram, then it is wild.
\end{theorem}
\begin{proof}
 First of all, we note that if the figure eight is a subdiagram, then either it is a connected component, or the the needle is also a subdiagram. If the figure eight is a connected component, we are finished by Theorem \ref{cor:closeddiagramswild}. So it suffices to show that if $T$ contains the open claw or needle as a subdiagram, $T$ is wild. 
 
 Let $S=(U,F)$ be the subdiagram of $T$ that is either the open claw or the needle. We consider the diagram $T_S$ as before, recalling that $\mcR(T_S)$ is equivalent to full subcategory of $\mcR(T)$ by Lemma \ref{lem:splitting}. We let $r=\{(e,1)|\; e\notin F\}$; we consider the map $\varphi:\mcR(S)\to \mcR(T_S)^r$ constructed for the proof of Lemma \ref{lem:subdiagram}, noting that it is well defined even if $S$ and $T$ have dangling wires.
 
 We know from the proof of Lemma \ref{lem:subdiagram} that the subgroup $\GL(\overline{H})\subseteq \GL(T_S)$, defined by acting by the identity on the wires $H$, preserves isomorphism classes of representations of $S$ under the image of $\varphi$. Let $\GL(H)$ be the subgroup of $\GL(T_S)$ that acts as the identity on all the wires in $E\setminus H$. We note that $\GL(T_S)\cong \GL(H)\times \GL(\overline{H})$. So we need to show that $\GL(H)$ preserves isomorphism classes of representations of $S$ under the image of $\varphi$. 
 
 If $H\subseteq F$, then this is trivial. So let us look at the action of $\GL(H)$ on the dangling wires outside of $F$, which acts by multiplication by scalars on each of the wires. By Fact \ref{fact:isom}, we need only consider the action on a single dangling wire outside of $F$. But as $U=\{c\}$, this is isomorphic to multiplying the tensor associated to $c$ by a scalar $\alpha$, again by Fact \ref{fact:isom}. Lastly, we note that if we have a representation $R(S)$, $\alpha R(S)\cong R(S)$ for all $\alpha\in\C^\times$ as both the needle and open claw have a dangling wire. So $\GL(H)$ preserves isomorphism classes of representations of $S$ in the image of $\varphi$.
 
 Thus we have that the map $\varphi$ realizes $\mcR(S)$ as a category equivalent to a full subcategory of $\mcR(T)$. Since $\mcR(S)$ is wild, this implies that $\mcR(T)$ is wild.
\end{proof}

As a consequence of Theorem \ref{thm:wildcases}, we have that a tensor diagram is wild if it has vertex of degree $\ge 3$, where a loop is counted as a degree two edge. In the next section, we will prove that these are all the wild tensor diagrams.

\section{Finite and Tame Tensor Diagrams}\label{sec:tameandfinite}

In the previous subsection, we showed that any tensor diagram containing a vertex with degree at least three is wild. We claim that these are all the wild tensor diagrams. The remaining tensor diagrams to consider are those with the following underlying semi-graphs:

\begin{enumerate}[(a)]
 \item The path $P_n$, which consists of vertices labeled $1,\dots,n$ and wires connecting vertex $i$ to $i+1$ for $i\in\{0,\dots,n-1\}$.
 \item The open path $A^0_n$, which is the path $P_{n+2}$ with the two leaves removed.
 \item The half-open path $A^1_n$, which is the path $P_{n+1}$ with one of the leaves removed.
 \item The loop $J_n$ which is formed by taking the path $P_n$ and adding a wire between vertices $1$ and $n$.
\end{enumerate}

As was mentioned in the Preliminaries, a closed tensor diagram $T$ is a morphism in $\Hom(\unit,\unit)$ in a finitely generated free compact closed monoidal category $\mcF$. As such, any representation of a closed tensor diagram is a morphism in $\Hom(\C,\C)\cong \C$. That is to say, there is a fundamental invariant of the representation. In fact, this invariant is a $\GL(T)$ invariant polynomial. This implies that are infinitely many non-isomorphic representations of $T$ as there is an invariant that can take any value in $\C$. As such, the only tensor diagrams that are eligible to be finite are those whose underlying semi-graph is the open or half-open path.

\begin{lemma}\label{lem:finitecases}
 The tensor diagrams with underlying semi-graph $A^0_n$ and $A^1_n$ are finite.
\end{lemma}
\begin{proof}
 We shall first show that $\mcR(A^0_n)$ is equivalent to $A_{n+1}-$Mod, where $A_{n+1}$ is the quiver associated with the Dynkin diagram $A_{n+1}$, whose orientation we shall specify shortly. We may assume that all wires in $A^0_n$ are oriented in the same direction by Lemma \ref{lem:orientationindependence}; similarly for the quiver $A_{n+1}$. Let $R(A^0_n)\in\mcR(A^0_n)$. We construct a map $\varphi$ taking $R(A^0_n)$ to a representation $R_\varphi(A_{n+1})$.
 
 Let the wires in $A^0_n$ be labeled in sequence $e_1,\dots,e_{n+1}$. Let $q_{1},\dots,q_{n+1}$ be the vertices of $A_{n+1}$ in sequence. Then we define $R_\varphi(q_i)=R(e_i)$. Let $p_1,\dots,p_n$ be the arcs of $A_{n+1}$ listed in sequence and $v_1,\dots,v_n$ be the vertices of $A^0_n$ listed in sequence. Then define $R_\varphi(p_i)=R(v_i)$. It is clear that $\varphi$ defines a equivalence of categories, in fact an isomorphism. Since $A_{n+1}$ is finite by \cite{gabriel1972unzerlegbare}, $A^0_n$ is also finite.
 
 Now we show that $\mcR(A^1_n)$ is equivalent to a full subcategory of $A_{n+1}$-mod by constructing a map $\varphi$. Let $R(A^1_n)\in\mcR(A^1_n)$. As before, we assume the orientation of all the wires in $A^1_n$ and $A_{n+1}$ are in the same direction. We let $e_1,\dots,e_n$ be the wires of $A^1_n$ in sequence, with $e_1$ the dangling wire,  and $p_1,\dots,p_n$ be the arcs of $A_{n+1}$ in sequence. We let $v_1,\dots,v_n$ be the vertices of $A^1_n$ in sequence and $q_1,\dots,q_{n+1}$ be the vertices of $A_{n+1}$ in sequence. For $i\in[n]$, we define $R_\varphi(q_i)=R(e_i)$ and define $R_\varphi(q_{n+1})=\C$. Then for $i\in[n]$, we define $R_\varphi(p_i)=R(v_i)$.  This clearly shows that $\mcR(A^1_n)$ is a full subcategory of $A_{n+1}$-Mod and so $A^1_n$ is finite.
\end{proof}

We know that the tensor diagrams whose underlying semi-graph is the path $P_n$ or loop $J_n$ are not finite as they are closed. We know from Lemma \ref{lem:subdiagram} that it is sufficient to show that $J_n$ is tame as any tensor diagram with underlying semi-graph $\mcR(P_n)$ is equivalent to a full subcategory of a tensor diagram with underlying semi-graph $\mcR(J_n)$. 

\begin{lemma}\label{lem:tamecases}
 The category $\mcR(T)$, where $T$ has underlying semi-graph $J_n$, is equivalent to the category $\tilde{A}_{n-1}$-Mod, where $\tilde{A}_{n-1}$ is the affine Dynkin diagram.
\end{lemma}
\begin{proof}
 We construct a map $\varphi:\mcR(T)\to\tilde{A}_{n-1}$-Mod as follows. Let $R(T)\in \mcR(T)$. We may assume that the orientation of the wires of $J_n$ and $\tilde{A}_{n-1}$ are all in the same direction by Lemma \ref{lem:orientationindependence}. Let $v_1,\dots,v_n$ be the vertices of $J_n$ in sequence and $q_1,\dots,q_n$ be the vertices of $\tilde{A}_{n-1}$ in sequence. Similarly let $e_1,\dots,e_n$ be the wires of $J_n$ in sequence and $p_1,\dots,p_n$ be the arcs of $\tilde{A}_{n-1}$ in sequence. Then we define $R_\varphi(p_i)=R(v_i)$ and $R_\varphi(q_i)=R(e_i)$ for all $i\in[n]$. This clearly defines an equivalence, and indeed an isomorphism of the categories $\tilde{A}_{n-1}$-Mod and $\mcR(T)$. Since the quiver $\tilde{A}_{n-1}$ is tame (cf. \cite{brion2008representations}), so is $T$.
\end{proof}

\noindent This concludes our classification, giving us the following theorem:

\begin{theorem}\label{thm:trichotomy}
 A connected tensor diagram is
 \begin{enumerate}[(a)]
  \item finite if and only if it its underlying semi-graph is either $A^0_n$ or $A^1_n$,
  \item tame (but not finite) if and only if its underlying semi-graph is either $P_n$ or $J_n$,
  \item wild otherwise.
 \end{enumerate}
\end{theorem}
\begin{proof}
 This follows directly from Theorem \ref{thm:wildcases} and Lemmas \ref{lem:finitecases} and \ref{lem:tamecases}.
\end{proof}

\subsection{Classifying indecomposable representations of finite and tame tensor diagrams} In the proofs of Lemmas \ref{lem:finitecases} and \ref{lem:tamecases}, we saw that $\mcR(A^0_n)$ was isomorphic to $A_{n+1}$-Mod, where $A_{n+1}$ is the Dynkin diagram. Similarly $\mcR(J_n)$ was isomorphic to $\tilde{A}_n$-Mod, where $\tilde{A}_n$ is the affine Dynkin diagram. The indecomposable representations of any quiver whose underlying graphs are these Dynkin diagrams has been well worked out and gives a classification of the indecomposable representations of $A^0_n$ and $J_n$.

Let us first consider the the indecomposable representations of the quivers with underlying graph the Dynkin diagram $A_n$, which is the path on $n$ vertices. Given any orientation on $A_n$, we call a representation of $A_n$, $R(A_n)$, \emph{thin} if $\dim(R(v))=0,1$ for every vertex $v$. We say that the representation of $R(A_n)$ is \emph{connected} if the underlying graph of $A_n$ formed by deleting those vertices $v$ with $\dim(R(v))=0$ and arcs $e$ with $R(e)=0$, is connected.

\begin{theorem}[\cite{gabriel1972unzerlegbare}]\label{thm:thinandconnected}
 A representation of $A_n$ is indecomposable if and only if it is thin and connected.
\end{theorem}

This theorem follows from Gabriel's result that the indecomposable representation of a simply laced Dynkin diagram are in correspondence with the positive roots of the root system defined by the Dynkin diagram and the connection of these positive roots with the Tits form associated to a quiver. A short proof ot Theorem \ref{thm:thinandconnected} can be found here \cite{ringel2013representations}.

In the proof of Lemma \ref{lem:finitecases}, we showed that the representations of $A^1_n$ were in bijection with the representations of $A_{n+1}$ where one of the leaves of $A_{n+1}$ was forced to always be associated to $\C$. If $q_1,\dots,q_{n+1}$ are the vertices of $A_{n+1}$ in sequence, let us take the convention that it is $q_{n+1}$ that must be dimension one. This gives us the following corollary.

\begin{corollary}
 Up to isomorphism, the indecomposable representations of $A^1_n$ are in bijection with the thin and connected representations of $A_{n+1}$, $R(A_{n+1})$, where $\dim(R(q_{n+1}))=1$.
\end{corollary}
\begin{proof}
 This follows directly from the proof of Lemma \ref{lem:finitecases} and Theorem \ref{thm:thinandconnected}.
\end{proof}

We now look to the determining the indecomposable representations of $P_n$. Since it is a subdiagram of $J_n$, we know that its indecomposable representations must take the form of indecomposable representations of the quiver $\tilde{A}_{n+1}$. Looking at the proofs of Lemma \ref{lem:splitting} and Lemma \ref{lem:tamecases}, we see that the representations of $P_n$ correspond to representations of $\tilde{A}_{n+1}$ where at least one of the matrices associated to an arc has rank one.

We now recall the classification of indecomposable representations of $\tilde{A}_{n+1}$. Let us denote the vertices of $\tilde{A}_{n+1}$ in sequence by $q_1,\dots,q_{n+1}$, and its arcs in sequence by $p_1,\dots,p_{n+1}$. We will consider $\tilde{A}_{n+1}$ with the orientation such that the arc $p_i$ leaves $q_i$ and enters $q_{i+1}$. Let us consider a representation $R(\tilde{A}_{n+1})$, and let us define the matrix $L=\prod_{i=1}^{n+1}{R(p_{i})}\in\End(R(q_1))$.

We may assume that $L$ is in Jordan canonical form by performing the appropriate change of basis on $R(q_1)$. Let $\lambda$ be an eigenvalue of $L$ with eigenvector $v_1$. Then we define $v_{i+1}=R(p_i)v_i$ for $i=2,\dots,n$. If $\lambda=0$, then for some $i\in[n+1]$, $R(p_i)v_i=0$. Then consider the representation on $\tilde{A}_{n+1}$ given by $R'(p_j)=0$ and $R'(q_j)=\{0\}$ for $j\ne i$, and $R'(p_i)=R(p_i)$ and $R'(q_i)=\C v_i$. This is a simple subrepresentation of $R(\tilde{A}_{n+1})$ which we denote $V_{0,i}$.

Otherwise, if $\lambda\ne0$, then we get the following induced subrepresentation: $R'(p_i)=\C v_i$ for all $i\in[n+1]$ which is isomorphic the representation where $R''(p_i)=\C$ for all $i\in[n+1]$, $R''(q_i)=1$ for all $i\in[n]$, and $R''(q_{n+1})=\lambda$. This is also a simple representation which we denote $V_{\lambda}$.

Given a representation of $P_n$, we may view it as a representation of $\tilde{A}_{n+1}$ where the matrix $L$ is rank one. If $L$ is diagonalizable, then it has a single eigenvalue $c$ and it is isomorphic to one of the above simple representations. Otherwise, $L$ has only zero eigenvalues, and a single Jordan block of size two. 

Let $R(\tilde{A}_{n+1})$ be a representation where $L$ has this form. We know that $R(q_1)$ is two dimensional with basis vectors $v_1,w_1$ such that $Lv_1=0$ and $Lw_1=v_1$. We define $v_i$ and $w_i$ as before. The fact that $Lv_1=0$ implies that for some $i\in [n+1]$, $R(p_i)v_i=0$. For that same $i$, $R(p_i)w_i$ maps to a non-zero vector in $R(q_{i+1})$. We see that the dimension of all vector spaces $R(q_j)$ for $1\le j\le i$, there is a two-dimensional subspace with $v_j$ and $w_j$ forming a basis. For $j>i$, we have a subspace of $R(q_j)$ spanned by $w_j$. So we have a subrepresentation given by $R'(q_j)=\tn{span}\{v_j,w_j\}$ for $j\in[n+1]$ (noting that $v_j=0$ for $j>i$) with the induced maps sending $v_j\to v_{j+1}$ $w_j\to w_{j+1}$ associated to $R(p_i)$. It is clear that this module is indecomposable by the fact that the Jordan matrix of size two with eigenvalue zero cannot be decomposed as a non-trivial matrix direct sum. We denote this representation by $W_{i}$. This proves the following proposition.

\begin{proposition}
 Every indecomposable representation of $P_n$ is isomorphic to one of either $V_{0,i}$, $V_{\lambda}$, or $W_{i}$, for $\lambda\in\C\setminus\{0\}$, $i\in[n+1]$.
\end{proposition}

\section{Conclusion}

Tensor networks are an important and widely used tool in physics, computer science, and statistics. In many cases, one is faced with determining when two tensor networks can be related by an element of $\GL(T)$, where $T$ is the underlying tensor diagram. In this paper, we have shown that any classification scheme of orbits is an intractable problem if one considers arbitrary tensor networks on most tensor diagrams. The tensor diagrams that admit classifications are very basic and uninteresting, reducing to a small subset of the tame cases arising in quiver theory.

However, in applications, one often does not consider arbitrary tensor networks on a given diagram. For example, in computer science applications, the dimensions of the representation are limited so that every wire has a two dimensional vector space associated to it. In tensor network states, or the study of entanglement of density operators, one is interested in closed orbits rather than all orbits. 

While it may be that many of these problems will still be wild, a much larger subset may be tractable than those presented in the current work. It is completely unknown how a trichotomy theorem would manifest itself with these restrictions and the current work represents only the first step towards understanding the difficulty of these more common questions. As such, much more work in this direction is necessary.

\subsection*{Acknowledgments} The author would like to thank Llu\'is Vena for helpful discussions. The research leading to these results has received funding from the European Research Council under the European Union's Seventh Framework Programme (FP7/2007-2013) / ERC grant agreement No 339109.

\bibliographystyle{plain}
\bibliography{bibfile}

 \end{document}